\documentclass{amsart}
\usepackage[latin9]{inputenc}
\usepackage{textcomp}
\usepackage{enumitem}
\usepackage{amstext}
\usepackage{amsthm}
\usepackage{amssymb}
\usepackage{esint}
\usepackage[numbers]{natbib}

\makeatletter
\numberwithin{equation}{section}
\numberwithin{figure}{section}
\theoremstyle{plain}
\newtheorem{thm}{\protect\theoremname}[section]
\newlist{casenv}{enumerate}{4}
\setlist[casenv]{leftmargin=*,align=left,widest={iiii}}
\setlist[casenv,1]{label={{\itshape\ \casename} \arabic*.},ref=\arabic*}
\setlist[casenv,2]{label={{\itshape\ \casename} \roman*.},ref=\roman*}
\setlist[casenv,3]{label={{\itshape\ \casename\ \alph*.}},ref=\alph*}
\setlist[casenv,4]{label={{\itshape\ \casename} \arabic*.},ref=\arabic*}
\theoremstyle{plain}
\newtheorem{question}[thm]{\protect\questionname}
\theoremstyle{remark}
\newtheorem{acknowledgement}[thm]{\protect\acknowledgementname}

\usepackage{latexsym}\usepackage{epsfig}

\theoremstyle{definition}

\makeatother

\providecommand{\acknowledgementname}{Acknowledgement}
\providecommand{\casename}{Case}
\providecommand{\questionname}{Question}
\providecommand{\theoremname}{Theorem}

\begin{document}
\title{A Note On Rainbow 4-Term Arithmetic Progression}
\author{Subhajit Jana}
\address{Department of Mathematics, Jadavpur University, Jadavpur-700032, West
Bengal, India.}
\email{suja12345@gmail.com}
\author{Pratulananda Das}
\address{Department of Mathematics, Jadavpur University, Jadavpur-700032, West
Bengal, India.}
\email{pratulananda@yahoo.co.in}
\begin{abstract}
Let $[n]=\{1,\,2,...,\,n\}$ be colored in k colors. A rainbow AP($k$)
in {[}n{]} is a k term arithmetic progression whose elements have
diferent colors. Conlon, Jungic and Radoicic \citep{key-10} had shown
that there exists an equinumerous $4$-coloring of $[4n]$ which happens
to be rainbow $AP(4)$ free, when n is even and subsequently Haghighi
and Nowbandegani \citep{key-8} shown that such a coloring of {[}4n{]}
also exists when $n>1$ is odd. Based on their construction, we shown
that a balanced $4$-coloring of $[n]$ ( i.e. size of each color
class is at least $\left\lfloor n/4\right\rfloor $ ) actually exists
for all natural number $n$. Further we established that for nonnegative
integers $k\geq3$ and $n>1$, every balanced k-coloring of $[kn+r]$
with $0\leq r<k-1$, contains a rainbow $AP(k)$ if and only if $k=3$.
In this paper we also have discussed about rainbow free equinumerous
$4$-coloring of $\mathbb{Z}_{n}$.
\end{abstract}

\maketitle

\section{Introduction}

The well known van der Waerden theorem in Ramsey theory states that,
for every natural numbers $k$ and $t$ and sufficiently large N,
every k-colouring of $[N]$ contains a monochromatic arithmetic progression
of length $t$. Motivated by this result, Radoicic had conjectured
that every equinumerous 3-colouring of $[3n]$ contains a 3-term rainbow
arithmetic progression, i.e., an arithmetic progression whose terms
are colored with distinct colors. 

In 1916, Schur \citep{key-12} had observed that for every $k$, if
$n$ is sufficiently large, then every $k$-coloring of $[n]:={1,...,n}$
contains a monochromatic solution of the equation $x+y=z$. More than
seven decades later, Alekseev and Savchev \citep{key-1} considered
what Bill Sands calls an un-Schur problem \citep{key-7}. They proved
that for every equinumerous 3-colouring of $[3n]$ (i.e., a coloring
in which different color classes have the same cardinality), the equation
$x+y=z$ has a solution with $x$, $y$ and $z$ belonging to different
color classes. Such solutions are now known as rainbow solutions.
E. and G. Szekeres subsequently asked the natural question as to whether
the condition of equal cardinalities for three color classes can be
weakened \citep{key-13}. Indeed, Sch\"{ }onheim \citep{key-11} was
able to show that for every 3-colouring of {[}n{]}, such that every
color class has cardinality greater than $n/4$, the equation $x+y=z$
has rainbow solutions. Moreover, he established that $n/4$ is optimal. 

Inspired by the un-Schur problem, Jungi\textasciiacute c et al. \citep{key-5}
sought a rainbow counterpart of van der Waerden\textquoteright s theorem.
Namely, the asked that given natural numbers $m$ and $k$, what conditions
on $k$-colorings of $[n]$ guarantee the existence of an $AP(m)$,
all of whose elements have distinct colors?

If every integer in $[n]$ is colored by the largest power of three
that divides it, then one immediately obtains a k-coloring of {[}n{]}
with $k\leq\left\lfloor log_{3}n\,+1\right\rfloor $ which does not
contain any rainbow AP(3). So, in views of Szemeredi\textquoteright s
theorem which states that a large cardinality in only one-color class
ensures the existence of a monochromatic AP($m$), one needs all color
classes to be \textquotedblleft large\textquotedblright{} to force
a rainbow $AP(m)$, where large means has positive upper density. 

In \citep{key-5}, it was proved that every 3-coloring of $[N]$ with
the upper density of each color class greater than $1/6$ yields a
rainbow $AP(3)$. Using some tools from additive number theory, they
later obtained similar (and stronger) results for 3-colorings of $\mathbb{Z}_{n}$
and $\mathbb{\mathbb{Z}}_{p}$, some of which were recently extended
by Conlon et al. \citep{key-3}. The more difficult interval case
was studied in \citep{key-10}, where it was shown that every equinumerous
3-coloring of $[3n]$ contains a rainbow $AP(3)$, that is, there
exist a rainbow solution of the equation $x+y=2z$. Finally, Axenovich
and Fon-Der-Flaass \citep{key-2} cleverly combined the previous methods
with some additional ideas to obtain a stronger result, conjectured
in \citep{key-5}.

What happens when there are more than 3 colors? Axenovich and Fon-Der-Flass
\citep{key-2} found an equinumerous $k$-coloring of $[2mk]$ which
contains no rainbow $AP(k)$, for every $k\ge5$. The most challenging
case seems to be $k=4$ . In 2007, Conlon et al. \citep{key-3} constructed
a rainbow free equinumerous 4-coloring of {[}4n{]}, whenever n is
even. In 2011, Haghighi, and Nowbandegani \citep{key-8} extended
the above result and proved that for every positive integer m, there
exists an equinumerous 4-coloring of {[}8m + 4{]} with no rainbow
$AP(4)$.

Based on their construction, we have been able to generalised the
above results for $k=4$, and established the following result and
the proof is presented in the next section.
\begin{thm}
For every positive integer $n\geq8$, there exists an equinumerous
$4$-coloring of $[n]$ with no rainbow $AP(4)$.
\end{thm}

As an outcome we have the following theorem, which in some sense finishes
the story of the existence of rainbow $AP(k)$ in equinumerous random
k-coloring of $[kn]$, $n>1$.
\begin{thm}
For positive integers $k\geq3$ and $n>1$, every equinumerous k-coloring
of $[kn+r]$ with $0\leq r\leq k-1$, contains a rainbow $AP(k)$
if and only if $k=3$. 
\end{thm}

\begin{proof}
The case k = 3 has already been dealt in \citep{key-2}. For $k=4$,
by theorem 1.1, a rainbow $AP(4)$ free 4-coloring of $[4n+r]$ with
$0\leq r\leq3$ is at hand for every n \textgreater{} 1. To construct
a 5-coloring of $[5n+r]$ with $0\leq r\leq4$, we can start with
a equinumerous $4$-coloring of {[}4n+r-1{]} if $r=4$, which has
no rainbow $AP(4)$ and then color $\{4n+4,\ldots,5n+4\}$ with the
fifth color and if $r<4$, then we can use a equinumerous $4$-coloring
of$[4n+r]$, which has no rainbow $AP(4)$ and then color $\{4n+r+1,\ldots,5n+r\}$
with the fifth color. Evidently, this equinumerous 5-coloring has
no rainbow $AP(5)$. One can inductively use this construction to
provide an equinumerous k-coloring of $[kn+r]$ with $0\leq r\leq k-1$
for every $k>5$, $n>1$, with no rainbow $AP(k)$.
\end{proof}
Before we present our construction to establish Theorem1.1, It is
important to note that in this coloring there is a color which appears
on consecutive integers. An important step in establishing the existence
of a rainbow $AP(3)$ in every equinumerous $3$-coloring of {[}n{]}
is in establishing the fact that at least one of the colors is recessive,
i.e., it does not appear on consecutive integers. Therefore, a natural
way to possibly force the existence of a rainbow $AP(4)$ is to assume
that every color is recessive. This is our motivation for the second
construction, presented in Section 3, where we have shown that every
equinumerous $4$-coloring of $\mathbb{Z}_{8}$ consists of a rainbow
$AP(4)$.

\section{Proof of Theorem 1.1. }

Let $A,B,C$ and $D$ be our four colors and denote by $\mathcal{A}$
the block $ABCC$ and by $\mathcal{B}$ the block $DDAB$. 

The equinumerous 4-coloring of $[n]$ in the proof of Theorem 2 in
\citep{key-3} is rainbow AP(4) free, whenever $n=8m$ and is given
as follows: 

\label{(1)}
\[
\mathcal{\underbrace{A\ldots\mathcal{A}}\underbrace{\mathcal{B}\ldots\mathcal{B}}}
\]
 where both the blocks $\mathcal{A}$ and $\mathcal{B}$ apears $m$-times
consecutively.

Our Construction for $[n]$, whenever $n=8m+1$, $m\in\mathbb{N}$,
is as follows:

\label{(2)}
\[
\underbrace{\mathcal{A\ldots\mathcal{A}}}\underbrace{\mathcal{B\ldots\mathcal{B}}}D
\]
 where both the blocks $\mathcal{A}$ and $\mathcal{B}$ appear $m$-times
consecutively.

What remains is to check that this coloring of $[8m+1]$ is rainbow
$AP(4)$ free. 

To get a contradiction, let $t_{1}<t_{2}<t_{3}<t_{4}$ denote the
terms of a rainbow $AP(4)$ in (2) with common diference $d$. Obviously,
d \textgreater{} 1. Since (1) is rainbow $AP(4)$ free, we must have
$t_{4}=8m+1$. Since the left hand side (the frst $4m$ numbers) of
(2) is colored only by $A$, $B$ and $C$, and right hand side (
the last $4m+1$ numbers) of (2) is colored only by $A$, $B$ and
$D$. Therefore we must have$t_{1}\leq4m$. Now two cases can arise.
\begin{casenv}
\item First suppose $t_{1}<t_{2}\leq4m<t_{3}<t_{4}$. If $d\equiv0$ (mod
$4$), then $t_{3}$ and $t_{4}$ are both colored $D$. If $d\equiv1$
(mod $4$), then $t_{1}$ and $t_{3}$ are both colored $B$. If $d\equiv2$
(mod $4$), then $t_{2}$ and $t_{3}$ are both colored $A$. If $d\equiv3$
(mod $4$), then $t_{3}$ and $t_{4}$ are both colored $D$.
\item Next suppose $t_{1}\leq4m<t_{2}<t_{3}<t_{4}$. If $d\equiv0$ (mod
$4$), then $t_{3}$ and $t_{4}$ are both colored $D$. If $d\equiv1$
(mod $4$), then $t_{1}$ and $t_{3}$ are both colored $B$. If $d\equiv2$
(mod $4$), then $t_{2}$ and $t_{4}$ are both colored $D$. If $d\equiv3$
(mod $4$), then $t_{3}$ and $t_{4}$ are both colored $D$.
\end{casenv}
Construction for $[n]$, whenever $n=8m+2$, $m\in\mathbb{N}$, and
is given as follows:

\label{(3)}
\[
\underbrace{\mathcal{A\ldots\mathcal{A}}}\underbrace{\mathcal{B\ldots\mathcal{B}}}DB
\]
 where both the blocks $\mathcal{A}$ and $\mathcal{B}$ appear $m$-times
consecutively.

Now we need to check that this coloring of $[8m+2]$ is again rainbow
$AP(4)$ free. 

To get a contradiction, let $t_{1}<t_{2}<t_{3}<t_{4}$ denote the
terms of a rainbow $AP(4)$ in (3) with common diference $d$. Obviously,
$d>1$. Since (1), (2) are rainbow $AP(4)$ free, we must have $t_{4}=8m+2$.
Since the left hand side (the frst $4m$ numbers) of (3) is colored
only by $A$, $B$ and $C$, and right side ( the last $4m+2$ numbers)
of (3) is colored only by $A$, $B$ and $D$, therefore we must have
$t_{1}\leq4m$. Now, two cases can occur.
\begin{casenv}
\item First Let $t_{1}<t_{2}\leq4m<t_{3}<t_{4}$. If $d\equiv0$ (mod $4$),
then $t_{1}$ and $t_{2}$ are both colored $B$. If $d\equiv1$ (mod
$4$), then $t_{1}$ and $t_{2}$ are both colored $C$. If $d\equiv2$
(mod $4$), then $t_{2}$ and $t_{4}$ are both colored $B$. If $d\equiv3$
(mod $4$), then $t_{1}$ and $t_{3}$ are both colored $A$.
\item Next let $t_{1}\leq4m<t_{2}<t_{3}<t_{4}$. If $d\equiv0$ (mod $4$),
then $t_{2}$ and $t_{3}$ are both colored $D$. If $d\equiv1$ (mod
$4$), then $t_{2}$ and $t_{4}$ are both colored $B$. If $d\equiv2$
(mod $4$), then $t_{3}$ and $t_{4}$ are both colored $B$. If $d\equiv3$
(mod $4$), then $t_{2}$ and $t_{4}$ are both colored $B$.
\end{casenv}
Construction for $[n]$, whenever $n=8m+3$, $m\in\mathbb{N}$, and
is given as follows:

\label{(4)}
\[
\underbrace{\mathcal{A\ldots\mathcal{A}}}\underbrace{\mathcal{B\ldots\mathcal{B}}}DBA
\]
 where both the blocks $\mathcal{A}$ and $\mathcal{B}$ appear $m$-times
consecutively.

Now we need to check that this coloring of $[8m+3]$ is rainbow $AP(4)$
free. 

To get a contradiction, let $t_{1}<t_{2}<t_{3}<t_{4}$ denote the
terms of a rainbow $AP(4)$ in (4) with common diference $d$, $d>1$.
Since (1), (2), (3) are rainbow $AP(4)$ free, we must have $t_{4}=8m+3$.
Since the left part (the frst $4m$ numbers) of (4) is colored only
by $A$, $B$ and $C$, and right part ( the last $4m+3$ numbers)
of (4) is colored only by $A$, $B$ and $D$, therefore $t_{1}\leq4m$.
Now two cases can occur.
\begin{casenv}
\item We can have $t_{1}<t_{2}\leq4m<t_{3}<t_{4}$. If $d\equiv0$ (mod
$4$), then $t_{1}$ and $t_{2}$ are both colored $C$. If $d\equiv1$
(mod $4$), then $t_{2}$ and $t_{4}$ are both colored $A$. If $d\equiv2$
(mod $4$), then $t_{1}$ and $t_{4}$ are both colored $A$. If $d\equiv3$
(mod $4$), then $t_{1}$ and $t_{3}$ are both colored $B$.
\item Otherwise let, $t_{1}\leq4m<t_{2}<t_{3}<t_{4}$. If $d\equiv0$ (mod
$4$), then $t_{2}$ and $t_{3}$ are both colored $A$. If $d\equiv1$
(mod $4$), then $t_{2}$ and $t_{3}$ are both colored $D$. If $d\equiv2$
(mod $4$), then $t_{2}$ and $t_{4}$ are both colored $A$. If $d\equiv3$
(mod $4$), then $t_{1}$ and $t_{3}$ are both colored $B$.
\end{casenv}
One can easilly show that following colorings of $[n]$, whenever
$n=8m+3$, $m\in\mathbb{N}$:

\label{(4.1)}
\[
C\underbrace{\mathcal{A\ldots\mathcal{A}}}\underbrace{\mathcal{B\ldots\mathcal{B}}}DB
\]
 where both the blocks $\mathcal{A}$ and $\mathcal{B}$ appear $m$-times
consecutively is also rainbow $AP(4)$ free, whose proof is left to
the reader.

The equinumerous $4$-coloring of $[n]$ in the proof of Theorem 1.3
in \citep{key-8} which has been shown to be rainbow $AP(4)$ free,
whenever $n=8m+4$ is the following: 

\label{(*)}
\[
BAC\mathcal{\underbrace{A\ldots\mathcal{A}}\underbrace{\mathcal{B}\ldots\mathcal{B}}}D
\]
 where both the blocks $\mathcal{A}$ and $\mathcal{B}$ apears $m$-times
consecutively.

However we can also consider the following construction for $[n]$,
whenever $n=8m+4$, $m\in\mathbb{N}$:

\label{(5)}
\[
C\underbrace{\mathcal{A\ldots\mathcal{A}}}\underbrace{\mathcal{B\ldots\mathcal{B}}}DBA
\]
 where both the blocks $\mathcal{A}$ and $\mathcal{B}$ appear $m$-times
consecutively. 

This coloring of $[8m+4]$ also happens to be rainbow $AP(4)$ free. 

To get a contradiction, let $t_{1}<t_{2}<t_{3}<t_{4}$ denote the
terms of a rainbow $AP(4)$ in (5) with common diference $d$ with,
$d>1$. Since (1), (2), (3), (4) are rainbow $AP(4)$ free, we must
have $t_{1}=1$. Since the left side (the frst $4m+1$ numbers) of
(5) is colored only by $A$, $B$ and $C$, and right side ( the last
$4m+3$ numbers) of (5) is colored only by $A$, $B$ and $D$, therefore
$t_{4}>4m+1$. Now two cases can arise.
\begin{casenv}
\item First let, $t_{1}<t_{2}\leq4m+1<t_{3}<t_{4}$. If $d\equiv0$ (mod
$4$), then $t_{1}$ and $t_{2}$ are both colored $C$. If $d\equiv1$
(mod $4$), then $t_{2}$ and $t_{4}$ are both colored $A$. If $d\equiv2$
(mod $4$), then $t_{2}$ and $t_{3}$ are both colored $B$. If $d\equiv3$
(mod $4$), then $t_{1}$ and $t_{2}$ are both colored $C$.
\item Otherwise, $t_{1}<t_{2}<t_{3}\leq4m+1<t_{4}$. If $d\equiv0$ (mod
$4$), then $t_{1}$ and $t_{2}$ are both colored $C$. If $d\equiv1$
(mod $4$), then $t_{2}$ and $t_{4}$ are both colored $A$. If $d\equiv2$
(mod $4$), then $t_{1}$ and $t_{3}$ are both colored $C$. If $d\equiv3$
(mod $4$), then $t_{1}$ and $t_{2}$ are both colored $C$.
\end{casenv}
We now come to the construction for $[n]$, whenever $n=8m+5$, $m\in\mathbb{N}$,
which can be taken as:

\label{(6)}
\[
CC\underbrace{\mathcal{A\ldots\mathcal{A}}}\underbrace{\mathcal{B\ldots\mathcal{B}}}DBA
\]
 where both the blocks $\mathcal{A}$ and $\mathcal{B}$ appear $m$-times
consecutively. 

What remains is to check that this coloring of $[8m+5]$ is rainbow
$AP(4)$ free. 

To get a contradiction, let $t_{1}<t_{2}<t_{3}<t_{4}$ denote the
terms of a rainbow $AP(4)$ in (6) with common diference $d$. Obviously,
d \textgreater{} 1. Since (1), (2), (3), (4), (5) are rainbow $AP(4)$
free, we must have $t_{1}=1$. Since the left part (the first $4m+2$
numbers) of (6) is colored only by $A$, $B$ and $C$, and right
part ( the last $4m+3$ numbers) of (6) is colored only by $A$, $B$
and $D$, therefore $t_{4}>4m+2$. Again two cases can occur.
\begin{casenv}
\item First assume, $t_{1}<t_{2}\leq4m+2<t_{3}<t_{4}$. If $d\equiv0$ (mod
$4$) or $d\equiv1$ (mod $4$), then $t_{1}$ and $t_{2}$ are both
colored $C$. If $d\equiv2$ (mod $4$), then $t_{2}$ and $t_{3}$
are both colored $A$. If $d\equiv3$ (mod $4$), then $t_{2}$ and
$t_{4}$ are both colored $B$.
\item Next assume, $t_{1}<t_{2}<t_{3}\leq4m+2<t_{4}$. If $d\equiv0$ (mod
$4$) or $d\equiv1$ (mod $4$), then $t_{1}$ and $t_{2}$ are both
colored $C$. If $d\equiv2$ (mod $4$), then $t_{1}$ and $t_{3}$
are both colored $C$. If $d\equiv3$ (mod $4$), then $t_{2}$ and
$t_{4}$ are both colored $B$.
\end{casenv}
Construction for $[n]$, whenever $n=8m+6$, $m\in\mathbb{N}$, can
be given as follows:

\label{(7)}
\[
BCC\underbrace{\mathcal{A\ldots\mathcal{A}}}\underbrace{\mathcal{B\ldots\mathcal{B}}}DBA
\]
 where both the blocks $\mathcal{A}$ and $\mathcal{B}$ appear $m$-times
consecutively. 

This coloring of $[8m+6]$ also happens to be rainbow $AP(4)$ free. 

To get a contradiction, let $t_{1}<t_{2}<t_{3}<t_{4}$ denote the
terms of a rainbow $AP(4)$ in (7) with common diference $d$. Obviously,
d \textgreater{} 1. Since (1), (2), (3), (4), (5), (6) are rainbow
AP(4) free, we must have $t_{1}=1$. Since the left side (the frst
$4m+3$ numbers) of (7) is colored only by $A$, $B$ and $C$, and
right side ( the last $4m+3$ numbers) of (7) is colored only by $A$,
$B$ and $D$, therefore $t_{4}>4m+3$. Now two cases can arise.
\begin{casenv}
\item Let, $t_{1}<t_{2}\leq4m+3<t_{3}<t_{4}$. If $d\equiv0$ (mod $4$),
then $t_{1}$ and $t_{2}$ are both colored $B$. If $d\equiv1$ (mod
$4$), then $t_{1}$ and $t_{3}$ are both colored $B$. If $d\equiv2$
(mod $4$), then $t_{1}$ and $t_{4}$ are both colored $B$. If $d\equiv3$
(mod $4$), then $t_{2}$ and $t_{3}$ are both colored $B$.
\item Next let, $t_{1}<t_{2}<t_{3}\leq4m+3<t_{4}$. If $d\equiv0$ (mod
$4$), then $t_{1}$ and $t_{2}$ are both colored $B$. If $d\equiv1$
(mod $4$), then $t_{2}$ and $t_{3}$ are both colored $C$. If $d\equiv2$
(mod $4$), then $t_{1}$ and $t_{3}$ are both colored $A$. If $d\equiv3$
(mod $4$), then $t_{1}$ and $t_{3}$ are both colored $A$.
\end{casenv}
Finally we take the Construction for $[n]$, whenever $n=8m+7$, $m\in\mathbb{N}$
as follows:

\label{(8)}
\[
ABCC\underbrace{\mathcal{A\ldots\mathcal{A}}}\underbrace{\mathcal{B\ldots\mathcal{B}}}DBA
\]
 where both the blocks $\mathcal{A}$ and $\mathcal{B}$ appear $m$-times
consecutively. 

To finish the proof, we claim that this coloring of $[8m+7]$ is rainbow
$AP(4)$ free. 

To get a contradiction, let $t_{1}<t_{2}<t_{3}<t_{4}$ denote the
terms of a rainbow $AP(4)$ in (8) with common diference $d$. Obviously,
d \textgreater{} 1. Since (1), (2), (3), (4), (5), (6), (7) are rainbow
$AP(4)$ free, we must have $t_{1}=1$. Since the left side (the frst
$4m+4$ numbers) of (8) is colored only by $A$, $B$ and $C$, and
right side ( the last $4m+3$ numbers) of (8) is colored only by $A$,
$B$ and $D$, therefore $t_{4}>4m+3$. Now two cases can occur.
\begin{casenv}
\item First let, $t_{1}<t_{2}\leq4m+3<t_{3}<t_{4}$. If $d\equiv0$ (mod
$4$), then $t_{1}$ and $t_{2}$ are both colored $A$. If $d\equiv1$
(mod $4$), then $t_{1}$ and $t_{3}$ are both colored $A$. If $d\equiv2$
(mod $4$), then $t_{1}$ and $t_{4}$ are both colored $A$. If $d\equiv3$
(mod $4$), then $t_{1}$ and $t_{3}$ are both colored $A$.
\item Otherwise we have, $t_{1}<t_{2}<t_{3}\leq4m+3<t_{4}$. If $d\equiv0$
(mod $4$), then $t_{1}$ and $t_{2}$ are both colored $A$. If $d\equiv1$
(mod $4$), then $t_{2}$ and $t_{4}$ are both colored $B$. If $d\equiv2$
(mod $4$), then $t_{1}$ and $t_{3}$ are both colored $A$. If $d\equiv3$
(mod $4$), then $t_{2}$ and $t_{3}$ are both colored $A$.
\end{casenv}

\section{Nature Of Equinumerous 4-colorings of Certain $\mathbb{Z}_{n}$}

Because arithmetic progressions may \textquotedblleft wrap around\textquotedblright{}
in the cyclic group $\mathbb{Z}_{n}$, it is only natural if we consider
only $k$-APs that include $k$ distinct members of $\mathbb{Z}_{n}$.
An interesting fact is that each arithmetic progression in $[n]$
is always an arithmetic progression in $\mathbb{Z}_{n}$, and it has
much more APs of same length. That's why the study is meaningful.

Some useful notation: given a 4-coloring $\nu=(\nu(0),\nu(1),...,\nu(n-1))\in\{A,B,C,D\}^{k}$
of $\mathbb{Z}_{k}$, let $\overline{\nu}$ denote the $4$-coloring
of $\mathbb{N}$ such that for every $i\in\mathbb{N}$, $\overline{\nu}(i)=\nu(i\ (mod\ k))$.

To construct a rainbow free equinumerous 4-coloring of $\mathbb{Z}_{24k}$,
$k\in\mathbb{N}$, Colnon et al.( in section 3, \citep{key-3}) have
considered the following 4-coloring $\mu$ of $[n]$, $n=24m$, $m\in\mathbb{N}$:
for $i\in[n]$, let 
\[
\mu(i):=\begin{cases}
A & i\equiv3,6,9,16,18,20\ (mod\ 24)\\
B & i\equiv1,8,10,12,19,22\ (mod\ 24)\\
C & i\equiv5,7,13,15,21,23\ (mod\ 24)\\
D & i\equiv0,2,4,11,14,17\ (mod\ 24)
\end{cases}
\]

In other words, $\mu([n])$ is a prefix of $\nu(\mathbb{N})$ of length
n, where $\nu$ denotes the 4-coloring of $\mathbb{Z}_{24}$ given
by: 
\[
\nu:=(B,D,A,D,C,A,C,B,A,B,D,B,C,D,C,A,D,A,B,A,C,B,C,D)
\]
 It is immediately clear that every color class of $\mu$ has exactly
$6m$ elements, so $\mu$ is equinumerous. Moreover, no two consecutive
integers receive the same color. And they have shown that this coloring
of $\mathbb{Z}_{24}$ is rainbow free and this coloring can furthermore
be easily extend to $\mathbb{Z}_{24k}$ , for any $k\in\mathbb{N}$
by repeating block of this 24 colors $k$ consecutive times. In \citep{key-3}
the authors state their strong intution that this curtails any hope
of getting a $\mathbb{Z}_{n}$ analogue of the result for $AP(3)$
( $Conjecture.1.2$ in \citep{key-5}, which was proved by Axenovich
and Fon-Der-Flaass, $Theorem.3$ \citep{key-2}) for $AP(4)$s.

The main aim in this section is to construct a suitable example to
show that their intution may not be exactly correct. 

Let us consider $n=8$. We will show that every equinumerous $4$-coloring
of $\mathbb{Z}_{8}$ consists of a rainbow $AP(4)$.

Note that $\nu:=(A,B,C,C,D,D,A,B)$ is a rainbow free equinumerous
4-coloring of $[8]$. Any 4-tuple $(x,y,z,w)$, $x,y,z,w\in\mathbb{Z}_{8}$,
and a common difference $d\in\mathbb{Z}_{8}$, such that $y\equiv x+d(mod\ 8)$,
$z\equiv x+2d(mod\ 8)$, and $w\equiv x+3d(mod\ 8)$, with $\nu(x)$,
$\nu(y)$, $\nu(z)$, and $\nu(w)$, being pairwise distinct. Note
that any such $4$-tuple (with common difference $d$) yields $(w,z,y,x)$,
another $4$-tuple with the same property, whose (common) difference
is $8-d$. Hence, we can restrict our attention to $4$-tuples with
difference at most $\text{4}$. In any coloring in which either $4$
consecutive integers or $4$ odd numbers or $4$ even numbers recieve
$4$ distinct colors, then it is trivial that such a coloring of $\mathbb{Z}_{8}$
consist of a rainbow $AP(4)$. 

So we only need to show that the each of the other possible equinumerous
$4$-coloring of $\mathbb{Z}_{8}$ consists of a rainbow $AP(4)$.
With out loss of generality, let us assume that $\mu(1)=A$ and all
the other possible equinumerous $4$-coloring of $\mathbb{Z}_{8}$
can be constructed as given bellow: 

\[
\nu_{1}:=(A,A,B,B,C,D,C,D)
\]

\[
\nu_{2}:=(A,D,B,B,C,A,C,D)
\]

\[
\nu_{3}:=(A,D,B,D,C,B,C,A)
\]

\[
\nu_{4}:=(A,D,C,A,C,B,B,D)
\]

\[
\nu_{5}:=(A,D,C,D,C,B,B,A)
\]

\[
\nu_{6}:=(A,A,C,D,C,D,B,B)
\]

\[
\nu_{7}:=(A,D,C,D,B,B,C,A)
\]

\[
\nu_{8}:=(A,A,C,B,B,D,C,D)
\]

\[
\nu_{9}:=(A,D,C,A,B,B,C,D)
\]

Now one can observe that, for each $k\in\{1,2,\dots,9\}$, $\left\{ \overline{\nu}(i+3j)\right\} _{j=0}^{\infty}$,
$i\in[8]$ contains all the 4 colors ($A,B,C$ and $D$) in four consecutive
positions. This shows that every equinumerous 4-coloring of $\mathbb{Z}_{8}$
consists of a rainbow $AP(4)$.

So the imediate question that will arise very naturally is the following:
\begin{question}
Does there exist any rainbow-free equinumerous 4-coloring of $\mathbb{Z}_{8k}$,
$k\in\mathbb{N}\backslash\{1\}$? 
\end{question}

Another question that also comes very naturally due to Conlon et.
al. \citep{key-3} is the following:
\begin{question}
Does there exist any rainbow-free equinumerous 4-coloring of $\mathbb{Z}_{24k+r}$,
$k\in\mathbb{N}$, $r=8,16$?
\end{question}

\begin{acknowledgement}
The first author is thankful to DSKPDF (MA/20-21/0120) funded by UGC
and second author is thankful to SERB (DST) for the project CRG/2022/000264
during the tenure of which this work has been done.
\end{acknowledgement}

\end{document}